\documentclass{amsart} 
\usepackage{amsthm}
\usepackage{amssymb}
\usepackage{amsfonts}
\usepackage{graphicx}
\usepackage{ytableau}
\usepackage{todonotes}
\usepackage{quiver}
\usepackage{bbold}
\usepackage{longtable}
\usepackage{geometry}
\usepackage{hyperref}
\usepackage[utf8]{inputenc}

\newtheorem{theorem}{Theorem}
\newtheorem{prop}{Proposition}
\newtheorem{lem}{Lemma}
\newtheorem{defin}{Definition}

\newtheorem{sugg}{Suggestion}

\newtheorem{cor}{Corollary}

\newtheorem*{example}{Example}

\newcommand{\rect}{\operatorname{Rect}}
\newcommand{\evil}{\operatorname{Evil}}

\newcommand{\st}{\operatorname{St}}
\newcommand{\ps}{\operatorname{ParSeq}}

\newcommand{\supp}{\operatorname{supp}}
\newcommand{\lies}{\mathfrak{sl}}
\newcommand{\An}{\mathcal{A}^{(1)}_n}
\newcommand{\Ao}{\mathcal{A}^{(1)}_{n+1}}
\newcommand{\Ai}{\operatorname{AI}}
\newcommand{\LL}{\operatorname{LL}}
\newcommand{\LR}{\operatorname{LR}}
\newcommand{\RL}{\operatorname{RL}}
\newcommand{\RR}{\operatorname{RR}}

\usepackage{tikz}
\usetikzlibrary{automata, positioning, arrows}
\tikzset{
->,
node distance=3cm,
every state/.style={thick, fill=gray!10},
initial text=$ $, 
}

\binoppenalty=\maxdimen
\relpenalty=\maxdimen

\usepackage[
backend=biber,
style=alphabetic,
sorting=nyt
]{biblatex}

\renewbibmacro{in:}{}
\DeclareFieldFormat{pages}{#1}
\DeclareFieldFormat[article,unpublished]{title}{#1} 

\title{A Bijection between Evil-avoiding and rectangular permutations}
\author{Katherine Tung}

\begin{document}

\maketitle

\begin{abstract} 
Evil-avoiding permutations, introduced by Kim and Williams in 2022, arise in the study of the inhomogeneous totally asymmetric simple exclusion process. Rectangular permutations, introduced by Chirivì, Fang, and Fourier in 2021, arise in the study of Schubert varieties and Demazure modules. Taking a suggestion of Kim and Williams, we supply an explicit bijection between evil-avoiding and rectangular permutations in $S_n$ that preserves the number of recoils. We encode these classes of permutations as regular languages and construct a length-preserving bijection between words in these regular languages. We extend the bijection to another Wilf-equivalent class of permutations, namely the $1$-almost-increasing permutations, and exhibit a bijection between rectangular permutations and walks of length $2n-2$ in a path of seven vertices starting and ending at the middle vertex. 
\end{abstract}
\section{Introduction}\label{sec:intro}

A permutation $\pi$ \emph{contains} a permutation $\sigma$ as a pattern if some subsequence of the values of $\pi$ has the same relative order as all of the values of $\sigma$. Otherwise, $\pi$ \emph{avoids} $\sigma$. Two classes of pattern-avoiding permutations, called \textit{evil-avoiding} and \textit{rectangular} permutations, have been of recent interest due to their algebraic significance. Kim and Williams gave the following definition of evil-avoiding permutations.

\begin{defin}
\textup{\cite{KW}} A permutation that avoids the patterns $2413, 4132, 4213$ and $3214$ is called \emph{evil-avoiding}.  \footnote{Kim and Williams called these permutations \textit{evil-avoiding} because if we replace I by 1, E by 2, L by 3, and V by
4, then EVIL and its anagrams VILE, VEIL and LEIV become the four patterns 2413, 4132, 4213 and 3214 \cite{KW}.
(Leiv is a name of Norwegian origin meaning “loaf.” Even though LIVE is an anagram of EVIL, an evil-avoiding permutation does not necessarily avoid the pattern 3142. )}
\end{defin}

For example, the permutation 1674523 is evil-avoiding, but the permutation 562314 is not evil-avoiding since, for instance, the subsequence 5214 reduces to the forbidden pattern 4213. Chirivì, Fang, and Fourier gave the following definition of rectangular permutations.

\begin{defin}
\textup{\cite{CFF}} A permutation that avoids the patterns $2413, 2431, 4213,$ and $4231$ is called \emph{rectangular}. 
\end{defin}

Both evil-avoiding permutations and rectangular permutations are enumerated by the same sequence, which begins $1,2,6,20,68,232,\ldots$ and appears in the OEIS as sequence A006012 \cite{OEIS}. The enumeration of evil-avoiding permutations is done by Kim and Williams in \cite{KW}, and the enumeration of rectangular permutations is done by Biers-Ariel and Chirivì, Fang, and Fourier in \cite{BA} and \cite{CFF}.

\begin{theorem} \label{thm:evil-count}
\textup{{{\cite[Proposition 1.14]{KW}}}} The number of evil-avoiding permutations in $S_n$ satisfies the recurrence $e(1) = 1, e(2) = 2, e(n) = 4e(n - 1) + 2e(n - 2)$.
\end{theorem}

\begin{theorem} \label{thm:rect-count}
\textup{{{\cite[Theorem 8]{BA}}}, {{\cite[Corollary 8]{CFF}}}} The number of rectangular permutations in $S_n$ satisfies the same recurrence $r(1) = 1, r(2) = 2, r(n) = 4r(n - 1) + 2r(n - 2)$.
\end{theorem}

\begin{defin}
Two families $U$ and $V$ of pattern-avoiding permutations are called \emph{Wilf-equivalent} if for all $n, |S_n \cap U| = |S_n \cap V|.$ This is a \emph{trivial Wilf-equivalence} if $U$ avoids patterns in $P_U$ and $V$ avoids patterns in $P_V$, where for some symmetry (i.e., rotation or reflection) $\rho$ of the square, $P_U = \rho(P_V).$  
\end{defin}

Theorems \ref{thm:evil-count} and \ref{thm:rect-count} say that evil-avoiding and rectangular permutations are Wilf-equivalent. Since no symmetry takes the set of patterns $\{2413,4132,4213,3214\}$ to the set of patterns $\{2413,2431,4213,4231\},$ this Wilf-equivalence is nontrivial. Thus, the following question is natural.

\begin{sugg}\label{q:unfiltered-bij}
\textup{\cite{pc}} Find an explicit bijection between evil-avoiding and rectangular permutations in $S_n$.
\end{sugg}

In this paper, we exhibit such a bijection. In fact, our bijection preserves not only the size of a permutation but also the number of recoils, as motivated in \cite{KW} and explained below.

\begin{defin}
A \emph{recoil} of a permutation $\pi \in S_n$ is a value $i \in \{1,2,\ldots,n-1\}$ so that $i$ occurs after $i+1$ in $\pi$, which means $\pi^{-1}_i > \pi^{-1}_{i+1}.$ Equivalently, $i$ is a recoil of $\pi$ if $i$ is a descent of $\pi^{-1}$.
\end{defin}

Kim and Williams \cite{KW} enumerated evil-avoiding permutations of size $n$ with $k$ recoils. A quick computational check shows that for small $n,$ there are the same number of evil-avoiding permutations of length $n$ with $k$ recoils as rectangular permutations of length $n$ with $k$ recoils, thus motivating the following notion.

\begin{defin}
Two classes of pattern-avoiding permutations are \emph{strongly Wilf-equivalent} if for every $n$ and $k$ the classes have the same counts of permutations in $S_n$ with $k$ recoils.
\end{defin}

In \cite{KW}, Kim and Williams observed that there are several permutation classes enumerated by OEIS sequence A006012, including

\begin{enumerate}
    \item permutations $\pi \in S_n$ for which the pairs $(i, \pi_i)$ with $i < \pi_i$, considered as closed intervals $[i + 1, 
    \pi_i]$, do not overlap; equivalently, for each $i \in [n]$ there is at most one $j \le i$ with $\pi_j > i$,
    \item permutations on $\{1,2,\dots,n\}$ with no subsequence $abcd$ such that $bc$ are adjacent in position and $\max(a, c) < \min(b, d)$,
    \item rectangular permutations,
\end{enumerate}

hence the following suggestion, which may be regarded as a refinement of Suggestion \ref{q:unfiltered-bij}.

\begin{sugg}
\textup{ \cite{KW}}
Biject evil-avoiding permutations with $k$ recoils (where we let $k$ vary) with any of the above sets of permutations in $S_n$.

\end{sugg}

We construct such a refined bijection. 

\begin{theorem}
There is an explicit bijection between evil-avoiding and rectangular permutations in $S_n$ with $k$ recoils.
\end{theorem}

For example, the bijection sends the evil-avoiding permutation in Figure \ref{fig:evil} to the rectangular permutation in Figure \ref{fig:rect}.

\begin{figure}[ht] 
    \centering
    \includegraphics[scale = 0.7]{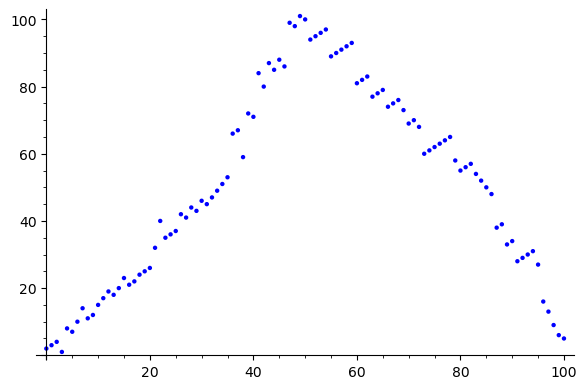}
    \caption{A plot of the evil-avoiding permutation \newline
    $[2, 3, 4, 1, 8, 7, 10, 14, 11, 12, 15, 17, 19, 18, 20, 23, 21, 22, 24, 25, 26, 32, 40, 35,
36, 37, 42, 41, 44, 43, 46, 45, 47, 49, 51, 53, 66, 67, 59, 72, 71, 84, 80, 87, 85, 88,
86, 99, 98, 101, 100, 94, 95, 96, 97, 89, 90, 91, 92, 93, 81, 82, 83, 77, 78, 79, 74,
75, 76, 73, 69, 70, 68, 60, 61, 62, 63, 64, 65, 58, 55, 56, 57, 54, 52, 50, 48, 38, 39,
33, 34, 28, 29, 30, 31, 27, 16, 13, 9, 6, 5]$.}
    \label{fig:evil}
\end{figure}
\pagebreak

\begin{figure}[ht] 
    \centering
    \includegraphics[scale = 0.7]{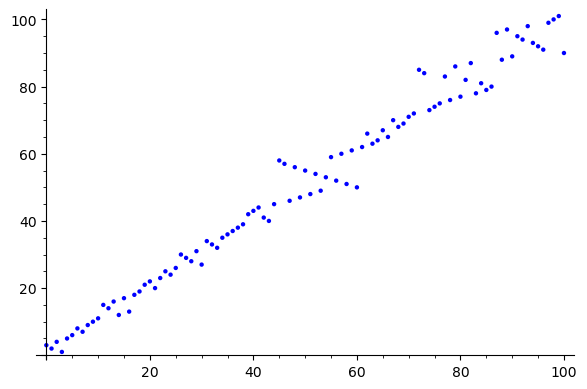}
    \caption{A plot of the rectangular permutation \newline
    $[3, 2, 4, 1, 5, 6, 8, 7, 9, 10, 11, 15, 14, 16, 12, 17, 13, 18, 19, 21, 22, 20, 23, 25, 24,
26, 30, 29, 28, 31, 27, 34, 33, 32, 35, 36, 37, 38, 39, 42, 43, 44, 41, 40, 45, 58, 57,
46, 56, 47, 55, 48, 54, 49, 53, 59, 52, 60, 51, 61, 50, 62, 66, 63, 64, 67, 65, 70, 68,
69, 71, 72, 85, 84, 73, 74, 75, 83, 76, 86, 77, 82, 87, 78, 81, 79, 80, 96, 88, 97, 89,
95, 94, 98, 93, 92, 91, 99, 100, 101, 90]$.}
    \label{fig:rect}
\end{figure}

A formal statement of the bijection appears in Section \ref{sec:the-bij}, Theorem \ref{thm:main}. We encode each evil-avoiding permutation and each rectangular permutation as elements of regular languages $L_{\evil}$ and $L_{\rect},$ respectively, then construct a bijection between these languages preserving the length and number of recoils of the corresponding permutations.

The significance of evil-avoiding permutations is rooted in Schubert calculus. There is a type of Markov chain called the \emph{asymmetric simple exclusion process} (ASEP) in which particles hop on a one-dimensional lattice subject to the condition that at most one particle may occupy a given site. The \textit{inhomogeneous totally asymmetric simple exclusion process} (inhomogeneous TASEP) is a type of ASEP where the sites $1, 2, \dots, n$ are arranged in a ring and the hopping rate depends on the weight of the particles. For more details on the inhomogeneous TASEP, see 
\cite{KW}, 
\cite{LW}, 
\cite{AL}, 
\cite{AM}, and 
\cite{Can}.
An interesting feature about ASEPs is that the steady-state probabilities at each site often contain Schubert polynomials, for reasons that are not entirely understood. In the case of the inhomogeneous TASEP, Kim and Williams demonstrate that a specialization of the steady-state probabilities at states corresponding to evil-avoiding permutations can be written as a ``trivial factor” times a product of (double) Schubert polynomials \cite{KW}. The number of Schubert polynomials in the steady-state formula is equal to the number of recoils in the corresponding evil-avoiding permutation \cite{KW}.

Rectangular permutations arise in the context of representation theory. The name ``rectangular" was coined by Chivrì, Fang, and Fourier in 2021 
and we briefly summarize its origin here; we direct the reader to their paper \cite{CFF} for more details.

Let $\Phi$ denote the root system of the Lie algebra $\lies_{n+1}$, let $\alpha_1, \dots, \alpha_n$ denote the simple roots of $\lies_{n+1}$, and let $$\Phi^+ = \{\alpha_{i,j} = \alpha_i + \alpha_{i+1} + \dots + \alpha_j \mid 1 \le i \le j \le n\}$$ denote the set of positive roots. The \textit{support} $\supp(\alpha_{i,j})$ of the root $\alpha_{i,j}$ is defined to be the subset $\{i, i+1, \dots, j\}$ of $\{1, 2, \dots, n\}$. The set $\Phi^+$ admits a poset structure with partial order relation given by $\alpha_{k,\ell} \le \alpha_{i,j}$ when $k \ge i$ and $ j \ge \ell.$ The meet $\alpha_{i,j} \wedge \alpha_{k,\ell}$ is defined by $\alpha_{\max(i,k),\min(j,\ell)}$ when it exists and dually, the join $\alpha_{i,j} \vee \alpha_{k,\ell}$ is defined by $\alpha_{\min(i,k),\max(j,\ell)}$. 

\begin{defin}\textup{\cite{CFF}}
A subset $A \subset \Phi^+$ is \textit{rectangular} if:
\begin{enumerate}
    \item it is triangular, i.e., for any $\alpha,\beta \in A$ such that $\supp(\alpha) \cup \supp(\beta)$ is a connected subset of $\{1, 2, \dots, n\}$, we have $\alpha \vee \beta \in A$; moreover, $\alpha \wedge \beta \in A$ if it exists,
    \item if $\alpha,\beta \in \Phi^+$ and $\alpha \wedge \beta,\alpha \vee \beta \in A$, then $\alpha,\beta \in A$ 
\end{enumerate} 

so that in the Hasse diagram for the subposet $(A,\le)$, we get the following ``rectangle":

% https://q.uiver.app/?q=WzAsNCxbMSwwLCJcXGFscGhhIFxcdmVlIFxcYmV0YSJdLFswLDEsIlxcYWxwaGEiXSxbMiwxLCJcXGJldGEiXSxbMSwyLCJcXGFscGhhIFxcd2VkZ2VcXGJldGEiXSxbMSwwXSxbMiwwXSxbMywxXSxbMywyXV0=
\[\begin{tikzcd}
	& {\alpha \vee \beta} \\
	\alpha && \beta \\
	& {\alpha \wedge\beta}
	\arrow[from=2-1, to=1-2]
	\arrow[from=2-3, to=1-2]
	\arrow[from=3-2, to=2-1]
	\arrow[from=3-2, to=2-3]
\end{tikzcd}\]
\end{defin}

The aforementioned rectangular condition translates nicely into a condition on permutations. 
\begin{defin}\textup{\cite{CFF}}
A permutation $\sigma$ is rectangular if for any 
${1 \le i < k < j< \ell \le n}$, the following conditions are equivalent:
\begin{enumerate}
    \item $\sigma_i > \sigma_j$ and $\sigma_k > \sigma_{\ell}$
    \item $\sigma_i > \sigma_{\ell}$ and $\sigma_k > \sigma_j$.
\end{enumerate}
\end{defin}

As justified in \cite{CFF}, an equivalent definition is that rectangular permutations avoid the patterns in ${\{2413,2431,4213,4231\}}.$

The rest of the paper is organized as follows. Section \ref{sec:prelim} gives preliminary definitions and constructions, and Section \ref{sec:the-bij} introduces operators for constructing rectangular and evil-avoiding permutations and explicitly state the bijection between the languages $L_{\rect}$ and $L_{\evil}$. Sections \ref{sec:regex-rect} and \ref{sec:regex-evil} establish the bijection between $L_{\rect}$ and rectangular permutations and between $L_{\evil}$ and evil-avoiding permutations. Section \ref{sec:examples} illustrates the bijection for a number of small permutations. Section \ref{sec:bael} discusses 1-almost-increasing permutations and bijects them with rectangular permutations. Section \ref{sec:paths} provides a bijection between a family of paths and rectangular permutations. Finally, Section \ref{sec:future-dir} discusses possible future directions, algebraic and enumerative.

\section{Preliminaries}\label{sec:prelim}

Let $S_n$ be the symmetric group on the set $\{1, \dots, n\}$. We say that a permutation $\pi \in S_n$ has \textit{size} $n$ and write $\pi$ in tabular form as $[\pi_1~\pi_2~\cdots~\pi_n]$. For $n\ge 0$, we denote by $e_n$ the identity in $S_n.$

We may grade the rectangular permutations by the size of the permutation, or double grade these permutations by the size and the number of recoils. 
Let $\rect$ denote the set of all rectangular permutations, let $\rect(n)$ denote the set of rectangular permutations in $S_n$, and let $\rect(n,k)$ denote the set of rectangular permutations in $S_n$ with $k$ recoils. The evil-avoiding permutations can be singly or doubly graded in the same way, and we define $\evil$, $\evil(n)$, and $\evil(n,k)$ analogously. 

Since our bijection involves regular expressions (regexes), we will define these here the way Sipser does \cite{Sipser}. Informally, 

\begin{defin}
A regular expression (regex) is a formula that describes a language (set of words) over an alphabet $\Sigma$. A regex is either $\emptyset$ or it is built recursively from individual elements of $\Sigma \cup \{\varepsilon \}$ using the operations of 
\begin{enumerate}
\item concatenation $(\cdot)(\cdot)$, 
\item OR $(\cdot| \cdot)$ (an alternative between two simpler patterns), or 
\item the unary Kleene star $(\cdot)^*$ ($R^*$ allows 0 or more repetitions of the pattern $R$). 
\end{enumerate}
\end{defin}

The symbol $\varepsilon$ represents the empty string. For example, $((a|\varepsilon)b)^*(a|\varepsilon)$ generates words in $\{a,b\}^*$ with no adjacent $a$'s. For clarity, we can also use the unary Kleene plus $(\cdot)^+$ ($R^+$ allows 1 or more repetitions of the pattern $R$). For example, $(a|b)^+$ describes any positive-length string of $a$'s and $b$'s.

It is equivalent to be able to describe a language (a subset of $\Sigma^*$) as the set of words generated by a regular expression and to say that there is a deterministic finite automaton (DFA) accepting the words \cite{Sipser}. 

\begin{defin}
A deterministic finite automaton $M$ is a 5-tuple $(Q, \Sigma, \delta, q_0, F)$ where
\begin{enumerate}
    \item $Q$ is a finite set of states,
    \item $\Sigma$ is an alphabet consisting of a finite set of input symbols,
    \item $\delta$ is a transition function from $Q \times \Sigma$ to $Q$,
    \item $q_0 \in Q$ is an initial state,
    \item $F \subset Q$ is a set of accepting states.
\end{enumerate}

If we let $w = a_1a_2\cdots a_n$ be a string with characters in $\Sigma$, then $M$ accepts $w$ if there is a sequence of states $r_0, r_1, \dots, r_n \in Q$ such that
\begin{enumerate}
    \item $r_0 = q_0$,
    \item $r_{i+1} = \delta(r_i, a_{i+1})$ for $i = 0, 1, \dots, n - 1$,
    \item $r_n \in F.$
\end{enumerate}
\end{defin}

We will consider many lengthening operators (denoted $\rho_{i,j}$ or $\gamma_{a,b}$) on permutations. These operators have domains that are subsets of $\cup_{n=0}^\infty S_n$. Applying a lengthening operator to a permutation in $S_n$ in the operator's domain produces a permutation in $S_{n+1}$.

\begin{defin}
The insertion operator $\rho_{i,j}$ takes in a permutation $\pi$, inserts the value $i$ at index $j$, and increments by $1$ all values in $\pi$ greater than or equal to $i$. Equivalently, the permutation matrix for the result has the original permutation matrix as the $(i,j)$ minor formed by deleting the $i$th row and $j$th column, which are both all $0$s except for a $1$ in the $(i,j)$ position. If $\pi\in S_n$ then $\rho_{i,j}(\pi)$ is defined when $1 \le i \le n+1$ and $1 \le j \le n+1.$ In tabular form, all entries at least $i$ are increased by $1$, then an $i$ is inserted between positions $j-1$ and $j$. If $\sigma = \rho_{i,j}(\pi),$ then for $1 \le k \le n+1,$

$$\sigma_k = \begin{cases} i & k=j\\
\pi_k & k<j \text{ and }\pi_k < i\\
\pi_k + 1 & k<j \text{ and } \pi_k \ge i \\
\pi_{k-1} & k>j \text{ and } \pi_{k-1} < i \\
\pi_{k-1} + 1 & k>j \text{ and } \pi_{k-1} \ge i.
\end{cases} $$
\end{defin}

\begin{example}
$\rho_{3,4}[5674321] = [67835421]$ because we increment by $1$ all values in $[5674321]$ greater than or equal to $3$ to get $[6785421]$ and then we insert $3$ into position $4$ to get $[67835421].$
\end{example}

We also define an operator $\gamma_{a,b}$ which does not increase the length of a permutation. It will be used to define one of the operators on evil-avoiding permutations. 

\begin{defin}
The shifting operator $\gamma_{a,b}$ takes in a permutation $\pi \in S_n$, removes the value at index $a$, and inserts it at index $b$ for $a > b.$ More formally, if $\tau = \gamma_{a,b}(\pi)$, then for $1 \le k \le n$,
$$
\tau_k =
\begin{cases}
\pi_k & k < b \\
\pi_a & k = b \\
\pi_{k-1} & b < k \le a \\
\pi_k & k > a.
\end{cases}
$$
\end{defin}
\begin{example}
$\gamma_{5,2}[5674321] = [5367421].$
\end{example}

Additionally, we will make use of an indicator variable $\mathbb{1}_S$ to describe the actions of our operators on individual values in permutations.
\begin{defin}
For a statement $S$, the indicator variable $\mathbb{1}_S$ is defined as follows:
$\mathbb{1}_S = 
\begin{cases}
1 & \text{if $S$ is true} \\
0 & \text{otherwise.}
\end{cases}
$
\end{defin}

\section{The bijection}\label{sec:the-bij}

In this section, we construct two regular languages of operators on permutations based on rectangular and evil-avoiding permutations, $L_{\rect}$ and $L_{\evil}$. We find regular expressions for these languages and construct a length-preserving bijection between them. In Sections \ref{sec:regex-rect} and \ref{sec:regex-evil}, we will establish that these languages are encodings of rectangular and evil-avoiding permutations respectively.

\subsection{Rectangular permutations}

We define four operators on rectangular permutations: $\psi_1, \psi_2, \psi_u,$ and $\psi_d.$ Each increases the size of the permutation by one, and we restrict the domains of some of the operators. Similar operators were introduced by Biers-Ariel to count permutations avoiding the patterns $1324, 1423, 2314,$ and $ 2413$ \cite{BA}, which are trivially Wilf-equivalent to rectangular permutations. 

The map $\psi_1$ is the insertion operator $\rho_{1,1}$. It increments by $1$ all values of the permutation, then inserts a $1$ at the beginning. The domain of $\psi_1$ is all rectangular permutations of size at least $0$. Applying $\psi_1$ does not change the number of recoils in a permutation.

The map $\psi_2$ is the insertion operator $\rho_{1,2}$ with a restricted domain.  
We restrict $\psi_2$ to rectangular permutations with a first element greater than $1$, that is, rectangular permutations of size at least $2$ that could not be in the image of $\psi_1$. (We are restricting the domain of $\psi_2$ to ensure that $\psi_2$ and $\psi_d$, defined later, have disjoint images, which helps us avoid ambiguous encodings.) This restriction means that $\psi_2(\pi)$ has the same number of recoils as $\pi$. 

The map $\psi_u$ applies $\rho_{\pi_1,1}$ to permutation $\pi$.
To ensure that $\psi_u$ and $\psi_1$ have disjoint images, we restrict $\psi_u$ to rectangular permutations of size at least $2$ whose first element is not $1$. Applying $\psi_u$ does not change the number of recoils of a permutation.

The map $\psi_d$ applies $\rho_{\pi_1+1,1}$ to permutation $\pi$. 

This operator is defined on all rectangular permutations of size at least $1$. Applying $\psi_d$ always increases the number of recoils by one.

Here are some examples of the four operators. For more such examples, see Table \ref{table:ex} in Section \ref{sec:examples}. 
\begin{example}
The permutation $\sigma = [3214]$ does not start with $1$ and is in the domain of all four of these operators, and the images are $\psi_1(\sigma)= [14325], \psi_2(\sigma) = [41325], \psi_u(\sigma) = [34215],$ and $\psi_d(\sigma) = [43215].$
\end{example} 
\begin{example}
The permutation $\tau = [126354]$ starts with a $1$, so it is not in the domain of $\psi_2$ or $\psi_u$ as these operations would duplicate the values of $\psi_d(\tau) = [2137465]$ and $\psi_1(\tau)=[1237465],$ respectively.
\end{example}

\begin{center}
\begin{tabular}{||c | c | c | c ||} 
 \hline
 \multicolumn{4}{|c|}{Operators on Rectangular Permutations} \\
 \hline
 Name & Domain & Insertion Description & Image of $\pi \in S_n$ \\ 
 \hline
 $\psi_1$ & $\rect$ & $\rho_{1,1}$ & $[1, \pi_1+1, \pi_2+1, \dots, \pi_n+1]$ \\ 
 \hline
 $\psi_2$ & $\rect \setminus \psi_1(\rect)$ & $\rho_{1,2}$ & $[\pi_1+1, 1, \pi_2+1, \dots, \pi_n+1]$ \\
 \hline
 $\psi_u$ & $\rect \setminus \psi_1(\rect)$  & $\rho_{\pi_1,1}$ & $[\pi_1, \pi_1+1, \pi_2 + \mathbb{1}_{\pi_2 > \pi_1}, \dots, \pi_n + \mathbb{1}_{\pi_n>\pi_1}]$ \\
 \hline
 $\psi_d$ & $\rect \setminus \{e_0\}$  & $\rho_{\pi_1+1, 1}$ &  $[\pi_1+1, \pi_1, \pi_2 + \mathbb{1}_{\pi_2 > \pi_1}, \dots, \pi_n + \mathbb{1}_{\pi_n>\pi_1}]$\\ 
 \hline
\end{tabular}
\end{center}
\vspace{1.0ex}
A composition in these operators can be abbreviated as a word in alphabet $A_r = \{1,2,u,d\}$, e.g., $\psi_u \circ \psi_d \circ \psi_1$ corresponds to the word $ud1.$

We prove in Section \ref{sec:regex-rect} that each rectangular permutation can be expressed uniquely as a composition of maps in $A_r^* = \{\psi_1, \psi_2, \psi_u, \psi_d\}$ applied to $e_0$. The procedures of $\psi_2$ and $\psi_d$ could be applied more widely, but since the domains are restricted to ensure uniqueness, they cannot be applied to any permutation in the image of $\psi_1$. Since only $\psi_1$ can be applied to the identity in $S_0$, the word encoding a permutation of positive length must end in $1$. Words in $A_r^*$ satisfying these restrictions form a regular language, which is a language that can be described with a regular expression.

\begin{lem} \label{lem:rect-regex}
Let $L_{\rect}$ be the language in $A_r^*$ of compositions of positive length of  $\{\psi_1,\psi_2,\psi_u,\psi_d \}$ applied to the permutation in $S_0$. 
\begin{enumerate}
\item The words of $L_{\rect}$ are precisely the words in $A_r^*$ that end in $1$ with no $21$ or $u1$.
\item A regular expression for $L_{\rect}$ is $(1^* (2|u)^* d)^* 1^+$.  
\end{enumerate}
\end{lem} 
\begin{proof}
Only $\psi_1$ can be applied to the element of $S_0$, so the first (rightmost) operator must be $\psi_1$, and the word must end in $1$. The domains of $\psi_2$ and $\psi_u$ are restricted to exclude precisely the images of $\psi_1$, so the substrings $21$ and $u1$ are forbidden. There are no other restrictions.

One way to produce the regular expression is to consider splitting a word encoding a rectangular permutation by the $d$'s (if any) into strings in $\{1,2,u\}^*.$ Since there are no $21$ or $u1$ substrings, the $1$s must be to the left of the $2$ and $u$ symbols (if any), so the strings between the $d$'s must be of the form $1^*(2|u)^*$. Further, there must be a terminal $1$, so the last substring must be all $1$'s. Hence, a regular expression generating $L_{\rect}$ is $(1^*(2|u)^*d)^*1^+.$
\end{proof}

\subsection{Evil-avoiding permutations}

We can similarly encode evil-avoiding permutations as compositions of size-increasing operators applied to the permutation $e_0$ in $S_0$. These operators are $\psi_p, \psi_q, \psi_r,$ and $\psi_s$, each of which has domain equal to a proper subset of evil-avoiding permutations. Compositions of these operators can be abbreviated as words in $A_e = \{p,q,r,s\}.$

The operator $\psi_p$ works the same way as $\psi_1$ and $\rho_{1,1}$, but is restricted to evil-avoiding permutations with at least one recoil (i.e., to nonidentity permutations).

The operator $\psi_q$ is defined on evil-avoiding permutations with at least one recoil. Let $\pi$ be a permutation of size $n$ in the domain of $\psi_q$. Let the least value involved in a recoil be $t$; precisely, $\pi^{-1}_t > \pi^{-1}_{t+1}$ but for all $i<t, $ we have $\pi^{-1}_i < \pi^{-1}_{i+1}$. Then calculate $\rho_{t+1,1}(\pi)$. 

For most $\pi$ in the domain, we will set $\psi_q(\pi) := \rho_{t+1,1}(\pi)$, but if $\pi$ takes a specific form, we will need to make a slight modification. We say that $\pi$ is $(a,b)$-\emph{sandwiched} (or \emph{sandwiched} for short) if it begins with a possibly empty increasing run of consecutive values $1, 2, \dots, a$ and ends with an increasing run $a + 1, a + 2, \dots, a + b$ for some nonnegative integer $a$ and positive integer $b$. If $\pi$ is $(a,b)$-sandwiched, then we remove the value in $\rho_{t+1,1}(\pi)$ at index $n - b + 2$ and insert it at index $a + 2$ to create a new permutation $\gamma_{n - b + 2, a + 2} \circ \rho_{t+1,1}(\pi)$. If $\pi$ has the above form, we set $\psi_q(\pi) := \gamma_{n - b + 2, a + 2} \circ \rho_{t+1,1}(\pi).$ (A straightforward term-by-term evaluation shows this is equivalent to setting $\psi_q(\pi) := (a + b + 1, 1, 2,\dots, a + 1, \pi_{a+1}+1, \pi_{a+2} + 1,\dots, \pi_{n-b}+1, a + 2, a + 3\dots, a + b).$) Otherwise, we set $\psi_q(\pi) = \rho_{t+1,1}(\pi).$ 

The operator $\psi_r$ is defined on all evil-avoiding permutations of size at least $1$.  
On an evil-avoiding permutation in $S_n$, this is $\rho_{1,n+1}.$

The operator $\psi_s$ is defined on evil-avoiding permutations that have $e_t$ as a suffix for some positive integer $t$. We allow there to be no values before the 1, so that any identity permutation of size at least $0$ is included in the domain of $\psi_s$.
This operator is a restriction of $\rho_{t+1,n+1} = \rho_{\pi_n,n+1}.$

Here are some examples of the four operators. For more such examples, see Section \ref{sec:examples}, Table \ref{table:ex}. 
\begin{example}
The permutation $\tau = [45123]$ is in the domain of all four operators $\psi_p,\psi_q,\psi_r$, and $\psi_s$. The images are $\psi_p(\tau) = [156234],$ $\psi_q(\tau) = [415623],$ $\psi_r(\tau) = [562341],$ and $\psi_s(\tau) = [561234]$. 
\end{example}
\begin{example}
The permutation $\nu=[21453]$ does not end in $e_t$ for $t \ge 1$ so it is not in the domain of $\psi_s$. The images under the other operators are $\psi_p(\nu) = [132564], \psi_q(\nu) = [231564],$ and $\psi_r(\nu)=[325641].$
\end{example}

\begin{center}
\begin{tabular}{||c | c | c | c ||} 
 \hline
 \multicolumn{4}{|c|}{Operators on Evil-Avoiding Permutations} \\
 \hline
 Name & Domain & Insertion Description & Image of $\pi \in S_n$ \\ 
 \hline
 $\psi_p$ & $\evil\setminus\{e_n | n\ge 0\}$ & $\rho_{1,1}$ & $[1, \pi_1+1, \pi_2+1, \dots, \pi_n+1]$ \\ 
 \hline
 $\psi_q$ & $\evil\setminus \{e_n | n \ge 0\}$  &  \begin{tabular}{@{}c@{}} $\rho_{t+1,1}$ \\ if $\pi$ not sandwiched \\ $\gamma_{n-b+2,a+2} \circ \rho_{t+1,1}$ \\ otherwise \end{tabular} & \begin{tabular}{@{}c@{}}$[t+1,\pi_1+\mathbb{1}_{\pi_1 > t}, \dots,\pi_n + \mathbb{1}_{\pi_n>t}]$ or \\  $[t+1,\pi_1+\mathbb{1}_{\pi_1 > t}, \dots,\pi_{a+1} + \mathbb{1}_{\pi_{a+1}>t}, $ \\ $\pi_{n-b+2} + \mathbb{1}_{\pi_{n-b+2}>t}, $\\$ \pi_{a+2} + \mathbb{1}_{\pi_{a+2} > t}, \dots, \pi_{n-b+1} + \mathbb{1}_{\pi_{n-b+1}>t},  $\\$ \pi_{n-b+3} + \mathbb{1}_{\pi_{n-b+3}>t}, \dots, \pi_n + \mathbb{1}_{\pi_n>t}]$\end{tabular} \\
 \hline
 $\psi_r$ & $\evil \setminus \{e_0\}$  & $\rho_{1,n+1}$ & $[\pi_1+1, \pi_2 + 1, \dots, \pi_n+1,1]$ \\
 \hline

 $\psi_s$ &  \begin{tabular}{@{}c@{}} $\{e_n | n\ge 0\} \cup$ \\ $  \bigcup_k \psi_s^k \circ \psi_r(\evil \setminus \{e_0\})$ \end{tabular} & $\rho_{\pi_n+1,n+1}$ &  $[\pi_1+1, \pi_2+1, \dots, 1, 2, \dots, \pi_n, \pi_n+1]$\\ 
 \hline
\end{tabular}
\end{center}
\medskip

In Section \ref{sec:regex-evil}, we show that every evil-avoiding permutation can be written uniquely as a composition of these operators applied to the permutation of $S_0$. Because of domain restrictions, not all words in $A_e^* = \{\psi_p, \psi_q, \psi_r, \psi_s \}$ are valid. 

\begin{lem} \label{lem:evil-regex}
Let $L_{\evil}$ be the language in $A_e^*$ of compositions of $\{\psi_p, \psi_q, \psi_r, \psi_s \}$ of positive length applied to the permutation in $S_0$. 
\begin{enumerate}
\item The words of $L_{\evil}$ are precisely the words in $A_e^*$ ending in $s$ with no $sp$ or $sq$ and neither $p$ nor $q$ before the terminal string of $s$'s.
\item A regular expression for $L_{\evil}$ is $((p|q)^* s^* r)^* s^+$. 
\end{enumerate}

\end{lem} 
\begin{proof}
The operator $\psi_s$ can be applied to the identity $e_0$ of $S_0$, while $\psi_p,\psi_q,$ and $\psi_r$ cannot, so a word in $L_{\evil}$ must end in $s$. The domain of $\psi_s$ is restricted so that it can only be applied after $\psi_r$ or $\psi_s,$ so the substrings $sp$ and $sq$ are forbidden. The operators $\psi_p$ and $\psi_q$ cannot be applied to the identity $e_n \leftrightarrow s^n$, so before the terminal string of $s$'s there can only be an $r$ (or nothing). There are no other restrictions.

To obtain the regular expression for such words, split on the $r$'s, if any, to produce a sequence of words in $\{p,q,s \}^*.$ The last word must be a string of $s$'s of positive length, hence it can be described by $s^+$. In each earlier word, the $s$'s must come after all of the $p$'s and $q$'s, so it can be described by $(p|q)^*s^*.$ Thus, a regular expression for $L_{\evil}$ is $((p|q)^*s^*r)^*s^+.$
\end{proof}

\subsection{Definition of the bijection}

\begin{theorem}\label{thm:main}
There is a length-preserving bijection $b: L_{\rect} \to L_{\evil}$ given by the below operations:
\begin{enumerate}
    \item substituting $2 \mapsto p$, $u \mapsto q$, $d \mapsto r$, and $1 \mapsto s$,
    \item reversing the prefix before the last $r$, if it exists. 
\end{enumerate}

This bijection has inverse $b^{-1}: L_{\evil} \to L_{\rect}$ given by:
\begin{enumerate}
    \item reversing the prefix before the last $r$, if it exists,
    \item substituting $p \mapsto 2$, $q \mapsto u$, $r \mapsto d$ and $s \mapsto 1$.
\end{enumerate}
\end{theorem}

The fact that $b$ is a bijection and $b^{-1}$ is its inverse follows immediately from the regular expressions $(1^*(u|2)^*d)^*1^+$ and $((p|q)^*s^*r)^*s^+$ for $L_{\rect}$ and $L_{\evil}$. 

\begin{example}
The rectangular permutation $$[4,1,2,5,6,3,9,8,10,7,11,13,12,15,14,17,18,19,20,16]$$ may be encoded by the word $$22uud1dud11d1d1uuud1$$ in $L_{\rect}$. To apply the map $b$ to this word, we substitute $2 \mapsto p$, $u \mapsto q$, $d \mapsto r$, and $1 \mapsto s$ to get $$ppqqrsrsrqrssrsrsqqqrs$$ and then reverse the prefix before the last $r$ to get $$qqqsrsrssrqrsrqqpprs$$ in $L_{\evil}$. After decoding this word, we get the evil-avoiding permutation $$[3, 4, 5, 1, 12, 11, 18, 19, 15, 16, 17, 20, 13, 14, 8, 9, 10, 6, 7, 2].$$
\end{example}

\section{A regular expression for rectangular permutations}\label{sec:regex-rect} 

We next establish the encoding of rectangular permutations of positive length as words in $L_{\rect}$, which is generated by the regular expression $(1^*(2|u)^*d)^*1^+.$ Throughout the proofs in this section, we say that two values $\pi_i,\pi_j$ in a permutation $\pi$ are \textit{adjacent} if $|\pi_i - \pi_j| = 1$. In particular, we do not require $|i - j| = 1.$

\begin{lem} 
If $\pi$ is rectangular, then any of the permutations $\psi_1(\pi), \psi_2(\pi), \psi_u(\pi),$ and $\psi_d(\pi)$ that are defined are also rectangular.
\end{lem}
\begin{proof}
These operators can be viewed as inserting an element into the first position for $\psi_1, \psi_u, \psi_d$ or second position for $\psi_2.$ We only need to demonstrate that there are no patterns 2413, 2431, 4213, and 4231 using the added element since any patterns without the added element would have been present in $\pi$. 

The added element of $\psi_1(\pi)$ is a leading 1. Any pattern involving this entry has a leading 1, but none of the forbidden patterns start with a 1, so if $\pi$ is rectangular, so is $\psi_1(\pi)$. 

The added elements of $\psi_u$ and $\psi_d$ are in the first position and are adjacent to the value in the second position. In all of the forbidden patterns, the first two values in the pattern are not adjacent. Thus, any forbidden pattern in $\psi_u(\pi)$ or $\psi_d(\pi)$ using the first value cannot also use the second value. Then, we can replace the first value with the second value to produce the same forbidden pattern, which means $\pi$ also would have to have that forbidden pattern. Thus, if $\pi$ is rectangular, so are $\psi_u(\pi)$ and $\psi_d(\pi)$ (if defined).

The added element of $\psi_2(\pi)$ is a 1 in the second position. If the 1 is used in a pattern, the pattern must have a 1 in the first or second position, but none of 2413, 2431, 4213, or 4231 have a 1 in the first or second position. Thus, if $\pi$ is rectangular, so is $\psi_2(\pi)$ (if defined).
\end{proof}

\begin{lem}  \label{lem:rect-uniquely-expressible}
For $n \ge 1$, every rectangular permutation in $\rect(n)$ is uniquely expressible as a map in the set $\{\psi_1, \psi_2,  \psi_u, \psi_d\}$ applied to an element of $\rect(n-1)$.  
\end{lem}
\begin{proof}
We claim that every rectangular permutation $\pi$ must have adjacent first two elements or a 1 in one of the first two positions. To see why, assume otherwise. Then the first two elements $\pi_1, \pi_2$ are not adjacent. Let $v$ be a value at least $\min(\pi_1,\pi_2)$ and at most $\max(\pi_1, \pi_2)$. Then the positions $\{1,2,\pi^{-1}_1, \pi^{-1}_v\}$ form a 2413, 2431, 4213, or 4231, which is a contradiction. Thus, the claim is shown.

If $\pi_1 = 1$, then $\pi$ is in the image of $\psi_1$. If $\pi_2=1$, then $\pi$ is in the image of $\psi_2$ if $\pi_1 > 2$ or in the image of $\psi_d$ if $\pi_1=2$. Otherwise, $|\pi_1-\pi_2|=1$. If $\pi_2 = \pi_1+1$ with $\pi_1>1$ then $\pi$ is in the image of $\psi_u$. If $\pi_2 = \pi_1-1$ then $\pi$ is in the image of $\psi_d.$ The restrictions of the domains for the four operators mean there is no ambiguity in deciding which operator produced $\pi$.  
\end{proof}

The following corollary follows immediately by induction.

\begin{cor}
Every rectangular permutation can be expressed uniquely as a composition of $\psi_1, \psi_2, \psi_u$, and $\psi_d$.
\end{cor}

Further, we can determine the number of recoils in a permutation from the corresponding composition. 

\begin{lem} \label{lem:rect-recoil}
The operators $\psi_1,\psi_2,$ and $\psi_u$ do not change the number of recoils of a permutation. The operator $\psi_d$ increases the number of recoils by $1$. 
\end{lem}

This lemma can be checked in a straightforward manner by examining each recoil $\pi_i$ in $\pi$ and the recoil's image under $\psi_1,\psi_2,$ $\psi_u$, or $\psi_d$. Under any of the first three maps, this process accounts for all of the recoils in the image permutation, but under the fourth map, there is one recoil unaccounted for in the image permutation, namely the one formed by $\pi_1$.

The following corollary is then immediate.

\begin{cor}
The permutations in $\rect(n,k)$ are encoded uniquely by the words in $L_{\rect}$ of length $n$ with $k$ $d$'s.
\end{cor}

\section{A regular expression for evil-avoiding permutations} \label{sec:regex-evil}

We now establish the encoding of evil-avoiding permutations of positive length as words in $L_{\evil}$, which is generated by the regular expression $((p|q)^*s^*r)^*s^+.$ 

The proof of this bijection involves some machinery in \cite{KW} used to analyze and count evil-avoiding permutations. We use slightly different notation than in \cite{KW} in some places, which we indicate explicitly.

\begin{defin}
\textup{\cite{KW}} Let $\st(n,k)$ be the set of permutations $\pi = (\pi_1, \dots, \pi_n)$ so that 
\begin{enumerate}
    \item $\pi_1 = 1$,
    \item $\pi$ is evil-avoiding,
    \item $\pi$ has $k$ recoils.
\end{enumerate}

An element of $\st(n,k)$ is called a \textup{$k$-Grassmannian permutation}.
\end{defin}

\begin{defin}
A \textup{partition} $\lambda = (\lambda_1, \dots, \lambda_k)$ of a nonnegative integer $L$ is a $k$-tuple of nonnegative integers so that $\sum_{i = 1}^k \lambda_i = L$ and $\lambda_1 \ge \lambda_2 \ge \dots \ge \lambda_k$. 
\end{defin}

\begin{defin}
\textup{\cite{KW}} A partition $\lambda$ is \textup{valid} for $n$ if $\lambda$ is contained inside a $\text{length}(\lambda) \times (n - \text{length}(\lambda))$ rectangle and is neither the empty partition nor the entire rectangle. 
\end{defin}

This definition does not require $\lambda$ to be a partition of $n$.

\begin{defin}
\textup{\cite{KW}} For $1 \le k \le n - 2$, let $\ps(n,k)$ denote the set of all sequences of partitions $(\lambda^1, \dots, \lambda^k)$ such that each $\lambda^i$ is valid for $n$, and for all $1 \le i \le k - 1,$ if $\ell$ is the smallest part of $\lambda^i$, then the first $(n - \ell)$ parts of $\lambda^{i + 1}$ are equal. If $k =0$, then $\ps(n,k)$ consists of one element, the empty sequence.
\end{defin}
The sets $\st(n,k)$ and $\ps(n,k)$ are natural to examine, as we explain. Observe that the sets $\evil(n-1,k)$ and $ \st(n,k)$ can be bijected by $f: \evil(n-1,k) \to \st(n,k)$ with $f(\pi) = (1, \pi_1 + 1, \dots, \pi_{n-1} + 1)$. The inverse of $f$ is given by $f^{-1}(\pi) = (\pi_2 - 1, \dots, \pi_n - 1)$.

In preparation for the next proposition, we recall that the \textit{Lehmer code} of a permutation $\pi$ is the tuple $c = (c_1, \dots, c_n)$ where $c_i = |\{j \mid j > i, \pi_j < \pi_i\}|.$ 

\begin{prop}
\textup{\cite{KW}} The sets $\st(n,k)$ and $\ps(n,k)$ are in bijection. To obtain the forward bijection $P$,\footnote{This operator is denoted $\Psi$ in \cite{KW}, but we call it $P$ instead to emphasize that it is a bijection and not an operator on permutations or partition sequences.} we associate to each $\pi \in \st(n, k)$ a sequence of partitions $P(\pi) = (\lambda^1, \dots, \lambda^k)$ as follows. Write the Lehmer code of $\pi^{-1}$ as $c(\pi^{-1}) = c = (c_1, \dots, c_n)$; since $\pi^{-1}$ has $k$ descents, $c$ has $k$ descents (i.e, values $c_i$ where $c_i > c_{i+1}$) in positions we denote by $a_1, \dots, a_k$ (where the position of a descent $c_i$ is the index $i$). Set $a_0 = 0$. For $1 \le i \le k$, define $\lambda^i = (\underbrace{n - a_i, \dots, n - a_i}_{a_i}) - (\underbrace{0, \dots, 0}_{a_{i-1}}, c_{a_{i-1} +1}, c_{a_{i-1} + 2}, \dots, c_{a_i})$ where subtraction of tuples is coordinatewise.
\end{prop}

\begin{example}
If $\pi \in \st(5,2)$ is $[13254]$, then the Lehmer code of $\pi^{-1}$ is $(0,1,0,1,0)$. This Lehmer code has two descents, one at index 2 and one at index 4, so we set $a_0 = 0, a_1 = 2, a_2 = 4.$ Then $$\lambda^1 = (5 - 2, 5 - 2) - (c_1, c_2) = (3, 3) - (0, 1),$$ which we subtract coordinatewise to get $(3, 2).$ We may also calculate $$\lambda^2 = (\underbrace{5 - 4, \dots, 5 - 4}_4) - (0, 0, c_3, c_4) = (1, 1, 1, 1) - (0, 0, 0, 1) = (1, 1, 1, 0).$$ We drop the trailing 0 for brevity. We then say that $P(\pi) = (\lambda^1, \lambda^2) = ((3,2), (1,1,1)).$
\end{example}

The inverse map $P^{-1}: \ps(n,k) \to \st(n,k)$ can be described as follows. Let $(\lambda^1, \dots, \lambda^k) \in \ps(n,k)$, and let $(f_1, \dots, f_k)$ be the sequence of first parts of $\lambda^1, \dots, \lambda^k$, i.e., $f_i = \lambda_1^i.$ Then $$(((\underbrace{f_1, \dots, f_1}_{n-f_1})-\lambda^1) + ((\underbrace{f_2, \dots, f_2}_{n-f_2})-\lambda^2) + \dots + ((\underbrace{f_k, \dots, f_k}_{n-f_k})-\lambda^k))$$ is the code of a permutation $\pi^{-1}$ of $\pi \in \st(n,k).$ We define $P^{-1}(\lambda^1, \dots, \lambda^k) = \pi.$

In \cite{KW}, Kim and Williams also define the injective maps $\Psi_1$, $\Psi_2$, and $\Phi_{i,k,n}$, whose images partition the set $\ps(n,k)$ into disjoint parts and make it simpler to count recursively.

\begin{defin}
\textup{\cite{KW}} For $k \ge 1,$ the map $\Psi_1: \ps(n-1,k) \to \ps(n,k)$ takes $(\lambda^1, \dots, \lambda^k)$ to $(\mu^1, \dots, \mu^k)$ where, for all $i \ge 1$, $\mu^i$ is obtained from $\lambda^i$ by duplicating its first part.
\end{defin}
\begin{example}
$\Psi_1((2), (1, 1)) = ((2, 2), (1, 1, 1)).$
\end{example}

\begin{defin}
\textup{\cite{KW}} For $k \ge 1,$ the map $\Psi_2: \ps(n-1,k) \to \ps(n,k)$ takes $(\lambda^1, \dots, \lambda^k)$ to $(\mu^1, \dots, \mu^k)$ where, for all $i \ge 2$, $\mu^i$ is obtained from $\lambda^i$ by duplicating its first part. For $i = 1$, we split into cases based on whether all parts of $\lambda^1 = (\lambda_1^1, \dots, \lambda_r^1)$ are equal. If so, then let $\mu^i = (\lambda_1^i+1, \lambda_1^i, \dots, \lambda_r^i).$ Otherwise, we define $\mu^1 = (\lambda_1^1 + 1, \lambda_2^1, \dots, \lambda_r^1).$ 
\end{defin}
\begin{example}
$\Psi_2((2), (1, 1)) = ((3, 2), (1, 1, 1))$ and $\Psi_2((3, 2), (1, 1, 1) = (4, 2), (1, 1, 1, 1).$
\end{example}

\begin{defin}
\textup{\cite{KW}} For $k + 1 \le i \le n - 1$, we define $\Phi_{i,k,n}: \ps(i,k-1) \to \ps(n,k)$ to be the map which takes $(\lambda^1, \dots, \lambda^{k-1})$ to $((i - 1), \mu^1, \dots, \mu^{k-1})$, where $\mu^j$ is obtained from $\lambda^j$ by duplicating the first part of $\lambda^j$ $n-i$ times. 
\end{defin}
\begin{example}
$\Phi_{4,2,5}((2, 1)) = ((3), (2, 2, 1)).$
\end{example}

\begin{prop}
    \textup{\cite{KW}} The set $\ps(n,k)$ equals the disjoint union $$\Psi_1(\ps(n-1,k)) \cup \Psi_2(\ps(n-1,k)) \cup \coprod_{i=k-1}^{n-1} \Phi_{i,k,n}(\ps(i,k-1)).$$ 
\end{prop}

Kim and Williams used this to prove a recurrence for the size of $\ps(n,k)$ in \cite{KW} Proposition 3.14. It also implies the following corollary:

\begin{cor}
    Every partition sequence in $\ps(n_0,k_0)$ is a composition of maps in $\{\Psi_1,\Psi_2, \Phi_{i,k,n} \}$ applied to an empty sequence in some $\ps(n_1,0).$
\end{cor}

\begin{theorem} \label{thm:evil-factor}
Every evil-avoiding permutation can be written uniquely as a composition of maps in the set $\{\psi_p,\psi_q,\psi_r,\psi_s\}$ applied to the identity in $S_0$.
\end{theorem}

Let $\psi_{i,k,n}$ be a map sending $(\pi_1, \dots, \pi_{i-1}) \to (\pi_1 + n - i, \pi_2 + n - i, \dots, \pi_{i - 1} + n - i, 1, 2, \dots, n - i)$. To prove this theorem, we provide alternative descriptions of maps $\psi_p,\psi_q,\psi_{i,k,n}$ which allow us to use a result from \cite{KW}.

\begin{theorem}
The maps $\psi_p,\psi_q,\psi_{i,k,n}$ may be written as
    \begin{align*}
    \psi_p(\pi) &= f^{-1} \circ P^{-1} \circ \Psi_1 \circ P \circ f(\pi) \\
    \psi_q(\pi) &= f^{-1} \circ P^{-1} \circ \Psi_2 \circ P \circ f(\pi) \\
    \psi_{i,k,n} (\pi) &= f^{-1} \circ P^{-1}\circ \Phi_{i,k,n} \circ P \circ f(\pi).
\end{align*}
\end{theorem}

\begin{proof}
To show that $\psi_p(\pi) = (1, \pi_1 + 1, \dots, \pi_{n-2} + 1) = f^{-1} \circ P^{-1} \circ \Psi_1 \circ P \circ f(\pi)$, we compute $ P\circ f (\pi)$ and compare with $P \circ f(1, \pi_1 + 1, \dots, \pi_{n-2} + 1)$. Let $c = (c_1, \dots, c_{n-1})$ be the Lehmer code of the inverse of $f(\pi)$, and let $a_1 \dots, a_{k}$ be the positions of the descents in $c$. Then the $i$th partition in the partition sequence $P \circ f(\pi)$ may be written as $$\lambda^i = (\underbrace{n - 1 - a_i, \dots, n - 1 - a_i}_{a_i}) - (\underbrace{0, \dots, 0}_{a_{i-1}}, c_{a_{i-1}+1}, \dots, c_{a_i}).$$

To get the Lehmer code of the inverse of $ f(1, \pi_1 + 1, \dots, \pi_{n-2} + 1)$ from $c$, we insert a 0 at the start of $c$ because the first entry of the Lehmer code of the inverse of a permutation counts the number of values in the permutation before the 1. Thus, the $i$th partition in the partition sequence $P \circ f(1, \pi_1 + 1, \dots, \pi_{n-2} + 1)$ is $$\mu^i = (\underbrace{n - (a_i + 1), \dots, n - (a_i + 1)}_{a_i + 1}) - (\underbrace{0, \dots, 0}_{a_{i-1} + 1}, c_{a_{i-1}+1}, \dots, c_{a_i}).$$

Comparing $\lambda^i$ with $\mu^i$, we see that the latter is the same as the former, except with an extra copy of $n - 1 - a_i$, which is the largest part of $\lambda^i$. (Observe that in the expression for $\lambda^i$, the subtrahend clearly has leading zeroes when $i \ge 2$, and when $i = 1$, we have $c_1 = 0$.) Note also that $\Psi_1$ acts on a partition sequence by mapping each partition to one with the first part duplicated, so $\Psi_1 \circ P \circ f(\pi) = P \circ f(1, \pi_1 + 1, \dots, \pi_{n-2} + 1)$, as desired.

Let $\pi_{ikn} = (\pi_1 + n - i, \pi_2 + n - i, \dots, \pi_{i - 1} + n - i, 1, 2, \dots, n - i)$. Now, we show that $$\psi_{i,k,n}(\pi) = \pi_{ikn} = f^{-1} \circ P^{-1} \circ \Phi_{i,k,n} \circ P \circ f(\pi).$$ (We consider $\psi_{i,k,n}$ before $\psi_q$ to make the exposition clearer.) Using a similar technique as for $\psi_p,$ we compute $P \circ f (\pi)$ and compare with $P \circ f(\pi_{ikn}).$ As computed earlier in this proof, the $j$th partition in the partition sequence $\pi$ may be written as $$\lambda^j = (\underbrace{i - a_j, \dots, i - a_j}_{a_j}) - (\underbrace{0, \dots, 0}_{a_{j - 1}}, c_{a_{j - 1} + 1}, \dots, c_{a_j})$$ with $c$ defined as above. 

To get the Lehmer code of the inverse of $f(\pi_{ikn})$, observe that 
\begin{align*}
    \pi_{ikn} &= (\pi_1 + n - i, \pi_2 + n - i, \dots, \pi_{i - 1} + n - i, 1, 2, \dots, n - i), \\
    f(\pi_{ikn}) &= (1, \pi_1 + n - i + 1, \pi_2 + n - i + 1, \dots, \pi_{i - 1} + n - i + 1, 2, \dots, n - i + 1), \text{and} \\
    (f(\pi_{ikn}))^{-1} &= (1, i + 1, i + 2, \dots, n, \pi_1^{-1} + 1, \dots, \pi_{i - 1}^{-1} + 1),
\end{align*}

so we have as our Lehmer code $(0, \underbrace{i - 1, \dots, i - 1}_{n - i}, c_2, c_3, \dots, c_i).$

Thus, the first partition in the partition sequence $P \circ f(\pi_{ikn})$ is 
\begin{align*}
    \mu^1 &= (\underbrace{n - (n - i + 1), \dots, n - (n - i + 1)}_{n - i  + 1}) - (0, \underbrace{i - 1, \dots, i - 1}_{n - i}) \\
    &= (i - 1, \underbrace{0, \dots, 0}_{n - i}) \\
    &= (i - 1).
\end{align*}

The $j$th partition, for $j > 1$, in the partition sequence $P \circ f(\pi_{ikn})$ is

\begin{align*}
    \mu^j &= (\underbrace{n - (a_{j-1} + n - i), \dots, n - (a_{j-1} + n - i)}_{a_{j-1} + n - i}) - (\underbrace{0, \dots, 0}_{a_{j-2} + n - i}, c_{a_{j-2} + 1}, \dots, c_{a_{j-1}}) \\
    &= (\underbrace{i - a_{j-1}, \dots, i - a_{j-1}}_{a_{j-1} + n - i}) - (\underbrace{0, \dots, 0}_{a_{j-2} + n - i}, c_{a_{j-2} + 1}, \dots, c_{a_{j-1}}).
\end{align*}

Comparing $\lambda^j$ with $\mu^{j+1}$, we observe that the latter is the same as the former, except with $n-i$ leading copies of $i - a_{j}$. Note that $\Phi_{i,k,n}$ takes a partition sequence $(\lambda^1, \lambda^2, \dots, \lambda^{k-1})$ to $((i - 1), \mu^1, \mu^2, \dots, \mu^{k-1}))$, where $\mu^j$ is obtained from $\lambda^j$ by duplicating the first part of $\lambda^j$ $n-i$ times, so the partition sequences match up as desired, and in fact $\Phi_{i,k,n} \circ P \circ f(\pi) = P \circ f (\pi_{ikn})$.

Finally, we consider $\psi_q$, and divide into cases based on whether $\pi$ is $(a,b)$-sandwiched. 

\noindent
\textbf{Case 1:} $\pi$ is $(a,b)$-sandwiched. 

We abbreviate $g(\pi) = (a + b + 1, 1, 2,\dots, a + 1, \pi_{a+1}+1, \pi_{a+2} + 1,\dots, \pi_{n-b-2}+1, a + 2, a + 3\dots, a + b)$. We write the Lehmer code of $f(\pi)^{-1}$ as $c = (c_1, c_2, \dots, c_{n-1}) $. where values $c_1, \dots, c_{a+1}$ are equal to 0 and values $c_{a+2}, \dots, c_{a + b + 1}$ are equal to $n - 2 - a - b.$ Now observe that $a_1 = a + b + 1$ (i.e., the first descent in $c$ is at index $a + b + 1$) because the value $a + b + 1$ appears at the end of $f(\pi)$ and the value $a + b + 2$ appears earlier (i.e., after the first ascending run and before the second one). Thus, the number of values that appear before and exceed $a + b + 1$ is greater than the number of values that appear before and exceed $a + b + 2.$ Equivalently, $c_{a + b + 1} > c_{a + b + 2}$.

We can compute
\begin{align*}
    \lambda^1 &= (\underbrace{n - 1 - a_1, \dots, n - 1 - a_1}_{a_1}) - (c_{1}, c_{2}, \dots, c_{a_1}) \\
    &= (\underbrace{n - 1 - (a + b + 1), \dots, n - 1 - (a + b + 1)}_{a + b + 1}) - (\underbrace{0, \dots, 0}_{a + 1}, \underbrace{n - 2 - a - b, \dots, n - 2 - a - b}_{b}) \\
    &= (\underbrace{n - 2 - a - b, \dots, n - 2 - a - b}_{a + 1}),
\end{align*} i.e., all nonzero parts of $\lambda^1$ are equal. 

For $i > 1$, we can write
\begin{align*}
    \lambda^i &= (\underbrace{n - 1 - a_i, \dots, n - 1 - a_i}_{a_i}) - (\underbrace{0, \dots, 0}_{a_{i - 1}}, c_{a_{i - 1} + 1}, \dots, c_{a_i}). 
\end{align*}

Now observe that the Lehmer code for $f(g(\pi))^{-1}$ can be written as $$(0, \underbrace{1, \dots, 1}_{a + 1}, \underbrace{n - 1 - a - b, \dots, n - 1 - a - b}_{b - 1}, 0, c_{a + b + 1}, c_{a + b + 2}, \dots, c_{n}).$$ From this, we can calculate that the first partition in the partition sequence $P \circ f(g(\pi))$ is

\begin{align*}
    \mu^1 &= (\underbrace{n - a_1, \dots, n - a_1}_{a_1}) - (0, \underbrace{1, \dots, 1}_{a + 1}, \underbrace{n - 1 - a - b, \dots, n - 1 - a - b}_{b - 1}) \\
    &= (\underbrace{n - (a + b + 1), \dots, n - (a + b + 1)}_{a + b + 1}) - (0, \underbrace{1, \dots, 1}_{a + 1}, \underbrace{n - 1 - a - b, \dots, n - 1 - a - b}_{b - 1}) \\
    &= (n - 1 - a - b)(\underbrace{n - 2 - a - b, \dots, n - 2 - a - b}_{a+1}),
\end{align*}

and the $i$th partition for $i > 1$ is

\begin{align*}
    \mu^i &= (\underbrace{n - (a_i + 1), \dots, n - (a_i + 1)}_{a_i + 1}) - (\underbrace{0,\dots,0}_{a_{i-1} +1}, c_{a_{i-1}+1}, \dots, c_{a_i}) \\
    &= (\underbrace{n - 1 - a_i, \dots, n - 1 - a_i}_{a_i + 1}) - (\underbrace{0,\dots,0}_{a_{i-1} +1}, c_{a_{i-1}+1}, \dots, c_{a_i}).
\end{align*}

Note that in the first subcase, when $\pi$ is $(a,b)$-sandwiched, $\Psi_2$ takes $(\lambda^1, \dots, \lambda^k)$ to $(\mu^1, \dots, \mu^k)$, where $\mu^1$ is $\lambda^1$ with the first part duplicated and then increased by 1, and for $i > 1$, $\mu^i$ is $\lambda^i$ with the first part duplicated. This is precisely what we found, so $\Psi_2 \circ P \circ f(\pi) = P \circ f(g(\pi)).$ It is a tedious though straightforward term-by-term arithmetic calculation to verify $$g(\pi) =[t+1,\pi_1+\mathbb{1}_{\pi_1 > t}, \dots,\pi_{a+1} + \mathbb{1}_{\pi_{a+1}>t},\pi_{n-b+2} + \mathbb{1}_{\pi_{n-b+2}>t}, $$ $$  \pi_{a+2} + \mathbb{1}_{\pi_{a+2} > t}, \dots, \pi_{n-b+1} + \mathbb{1}_{\pi_{n-b+1}>t},   \pi_{n-b+3} + \mathbb{1}_{\pi_{n-b+3}>t}, \dots, \pi_n + \mathbb{1}_{\pi_n>t}].$$

\noindent
\textbf{Case 2:} $\pi$ is not $(a,b)$-sandwiched. 

The proof is similar to the one for the previous case. As computed earlier, the $i$th partition in the partition sequence $P\circ f(\pi)$ may be written as $$\lambda^i = (\underbrace{n - 1 - a_i, \dots, n - 1 - a_i}_{a_i}) - (\underbrace{0, \dots, 0}_{a_{i-1}}, c_{a_{i-1}+1}, \dots, c_{a_i})$$ and, in particular, $$\lambda^1 = (\underbrace{n - 1 - a_1, \dots, n - 1 - a_1}_{a_1}) - (c_1, c_2, \dots, c_{t+1}).$$ Then the Lehmer code of $$f(t + 1, \pi_1 + \mathbb{1}_{\pi_1 > t}, \dots, \pi_{n - 2} + \mathbb{1}_{\pi_{n - 2} > t})^{-1}$$ is $$(0, c_2 + 1, c_3 + 1, \dots, c_{t+1}+1, 0, c_{t+2}, c_{t+3}, \dots, c_n)$$ with $a_1 = t + 1.$ We may then write $$\mu^1 = (\underbrace{n - a_1, \dots, n - a_1}_{a_1}) - (c_1, c_2 + 1, c_3 + 1, \dots, c_{t+1})$$ and then comparing $\mu^1$ with $\lambda^1$ part by part, we see that $\mu^1$ is identical to $\lambda^1$ except with the first part incremented by $1$. Since there is a 0 at index $t+2$ in the Lehmer code for $\pi$, offsetting the Lehmer code for $f(t + 1, \pi_1 + \mathbb{1}_{\pi_1 > t}, \dots, \pi_{n - 2} + \mathbb{1}_{\pi_{n - 2} > t})^{-1}$ back by 1, we get for $i > 1$,

\begin{align*}
    \mu^i &= (\underbrace{n - (a_i + 1), \dots, n - (a_i + 1)}_{a_i + 1}) - (\underbrace{0,\dots,0}_{a_{i-1} +1}, c_{a_{i-1}+1}, \dots, c_{a_i}) \\
    &= (\underbrace{n - 1 - a_i, \dots, n - 1 - a_i}_{a_i + 1}) - (\underbrace{0,\dots,0}_{a_{i-1} +1}, c_{a_{i-1}+1}, \dots, c_{a_i}).
\end{align*}

So in the second subcase, when $\pi$ is not $(a,b)$-sandwiched, $\Psi_2$ takes $(\lambda^1, \dots, \lambda^k)$ to $(\mu^1, \dots, \mu^k)$ where $\mu^1$ is $\lambda^1$ with the first part increased by 1, and for $i > 1$, $\mu^i$ is $\lambda^i$ with the first part duplicated. We find that $\Psi_2 \circ P \circ f(\pi) = P \circ f(t + 1, \pi_1 + \mathbb{1}_{\pi_1 > t}, \dots, \pi_{n - 2} + \mathbb{1}_{\pi_{n - 2} > t}).$ So the forms for $\psi_p,\psi_q,\psi_{i,k,n}$ stated at the start of this proposition are equivalent to the forms provided in Section \ref{sec:the-bij}. \end{proof}

\begin{prop}\label{prop:evil-disjoint}
If $k = 0$, then $\evil(n,k)$ contains only the identity permutation of length $n$. For $k>0$, $$\evil(n,k) = \psi_p(\evil(n-1,k)) \cup \psi_q(\evil(n-1,k)) \cup \coprod_{i=k-1}^{n-1} \!\psi_{i,k,n}(\evil(i,k-1))$$

and the terms on the right-hand side are pairwise disjoint.
\end{prop}
\begin{proof}
The first sentence of the proposition is true because every nonidentity permutation has at least 1 recoil. To see the second part, observe that from the proof of Proposition 3.14 in \cite{KW} we have a very similar result, namely 

$$\ps(n,k) = \Psi_1(\ps(n-1,k)) \cup \Psi_2(\ps(n-1,k)) \cup \coprod_{i=k-1}^{n-1} \!\Phi_{i,k,n}(\ps(i,k-1)),$$ where the terms on the right-hand side are disjoint.

After conjugating both sides by $P \circ f$, we get the desired result.
\end{proof}

\begin{prop}\label{prop:evil-recoil}
    The operators $\psi_p,\psi_q,$ and $\psi_s$ preserve the number of recoils. The operator $\psi_r$ increases the number of recoils by one.
\end{prop}
\begin{proof}
Observe that each value in $\psi_{v}(\pi)$ for any $v \in \{p,q,s\}$ depends on at most one value in $\pi$, and no two values in $\psi_{v}(\pi)$ depend on the same value in $\pi$. Throughout the proof, we say that a recoil $x$ at index $i$ in $\pi$ is mapped to a recoil $y$ at index $j$ in $\sigma = \psi_{v}(\pi)$ if the expression for $\sigma_j$ depends on $\pi_i$ and the evaluation $y$ of $\sigma_j$ at $x = \pi_i$ is a recoil.

To see that $\psi_p$ preserves the number of recoils, observe that a recoil $\pi_i$ in $\pi$ is mapped to recoil $\pi_i + 1$, and this map produces all recoils in $\psi_p(\pi).$ (The value 1 is first in $\psi_p(\pi)$ so cannot be a recoil.)

To see that $\psi_q$ preserves the number of recoils, we casework on whether $\pi$ is $(a,b)$-sandwiched. If not, observe that a recoil $\pi_i$ in $\pi$ is mapped to recoil $\pi_i + \mathbb{1}_{\pi_i > t}$ in $\psi_q(\pi)$ regardless of whether $\pi_i > t$ or $\pi_i < t$ or $\pi_i = t$, and this map produces all recoils in $\psi_q(\pi).$ If $\pi$ is $(a,b)$-sandwiched, then $\pi$ always contains the recoil $a + b$, and $\psi_q(\pi) = (a + b + 1, 1, 2,\dots, a + 1, \pi_{a+1}+1, \pi_{a+2} + 1,\dots, \pi_{n-b}+1, a + 2, a + 3\dots, a + b)$ necessarily contains $a+b$ as a recoil. Any other recoil in $\pi$ takes the form $\pi_i$ for $a+1 \le \pi \le n - b$ and is mapped to $\pi_i + 1.$ Now, the values $1, 2, \dots, a+b-1$ fail to be recoils in $\pi$ and the values $1, 2, \dots, a+b-1, a+b+1$ fail to be recoils in $\psi_q(\pi)$, so we have accounted for all possible recoils in $\pi$ and its image $\psi_q(\pi)$, and shown these recoils to be in bijection.

To see that $\psi_s$ preserves the number of recoils, observe that the smallest recoil in $\pi$ is $t$, and that a recoil $\pi_i$ in $\pi$ is mapped to $\pi_i + 1$ in $\psi_s(\pi).$ This map produces all recoils in $\psi_s(\pi).$ (The recoils in $\psi_s(\pi)$ all appear before the 1.)

To see that $\psi_r$ increases the number of recoils by one, observe that a recoil $\pi_i$ in $\pi$ is mapped to recoil $\pi_i + 1$ in $\psi_r(\pi)$, and this map produces all recoils in $\psi_r(\pi)$ except for the recoil 1.

\end{proof}

Now, making the observation that $\psi_{i,k,n} = \psi_s^{n-i-1} \circ \psi_r$, we can describe the set of words in $A_e = \{p,q,r,s\}$ that the set of evil-avoiding permtuations is in bijection with. This set was stated in Lemma \ref{lem:evil-regex}.

\begin{cor}
The words of $L_{\evil}$ encode evil-avoiding permutations uniquely.
\end{cor}

\subsection{Other consequences}

In a word of $L_{\rect}$ encoding a rectangular permutation, the number of recoils in the permutation equals the number of $d$'s. The bijection converts $d$'s to $r$'s, which count the recoils in the associated evil-avoiding permutation. Since the bijection preserves the number of recoils, the following consequence is immediate.

\begin{cor}\label{cor:bij}
There is an explicit bijection between rectangular permutations and evil-avoiding permutations in $S_n$ with $k$ recoils.
\end{cor}

Kim and Williams \cite{KW} proved a number of enumerative results about evil-avoiding permutations. The bijection between evil-avoiding permutations and $L_{\evil}$ allows us to provide some simpler proofs. 

Kim and Williams enumerated $\evil(n,k)$ using a recurrence and induction, but the structure of $L_{\evil}$ lets us enumerate this directly, which we can also do for $\rect(n,k)$.

\begin{theorem} \textup{\cite{KW} Corollary 3.15} For $k>0$,
$$ |\!\evil(n,k)| = \sum_{i=0}^{n-k-1} 2^i \binom{i+k-1}{k-1} \binom{n-i-1}{k}.$$
\end{theorem}
\begin{proof}
Permutations in $\evil(n,k)$ correspond to words in $L_{\evil}$ of size $n$ with $k$ $r$'s. The $r$'s split the word into $k+1$ pieces, where the last piece is all $s$'s. Let $i$ be the total number of $p$'s and $q$'s in the word.
These $i$ $p$'s and $q$'s can be distributed into the first $k$ pieces between the $r$'s in $\binom{i+k-1}{k-1}$ ways. The remaining $n-k-i$ positions must be $s$'s, and at least one $s$ must be in the last piece, so there are $n-k-i-1$ spares. 
We can allocate these among $k+1$ pieces in $\binom{n-i-1}{k}$ ways. For each such allocation of the letters to the $k+1$ pieces, the $s's$ must come at the end of the pieces, so the remaining choice is which of the $i$ $p|q$ positions are $p$'s and which are $q$'s. There are $2^i$ ways to make this choice. Thus, for each $i$, there are $2^i \binom{i+k-1}{k-1} \binom{n-i-1}{k}$ words in $\evil(n,k).$ Therefore, for $k>0,$
$$|\!\evil(n,k)| = \sum_i 2^i \binom{i+k-1}{k-1} \binom{n-i-1}{k}.$$
\end{proof}

Proposition \ref{prop:evil-disjoint} gives the following recurrence for $|\evil(n,k)|$ with $k>0:$

$$|\!\evil(n,k)| = 2 |\!\evil(n-1,k)| + \sum_{i=1} |\!\evil(n-i,k-1) |.  $$

There is a simpler recurrence, also mentioned in \cite{KW} in the proof of their Corollary 3.15.

\begin{theorem}
For $k>0$ and $n>1,$ 

$$|\!\evil(n,k)| = 3|\!\evil(n-1,k)|+|\!\evil(n-1,k-1)|-2|\!\evil(n-2,k)| .$$

\end{theorem}

\begin{proof}
This follows algebraically from the previous recurrence, but it also admits a relatively short combinatorial proof.

By Propositions \ref{prop:evil-disjoint} and \ref{prop:evil-recoil}, $\evil(n,k)$ is the disjoint union of the images of $\psi_p,\psi_q,$ and $\psi_s$ on the intersections of their domains with $\evil(n-1,k)$ together with the image of $\psi_r$ on $\evil(n-1,k-1).$

The operators $\psi_p$ and $\psi_q$ are defined on all of $\evil(n-1,k)$, and contribute ${2|\!\evil(n-1,k)|}$ permutations.

The operator $\psi_r$ is defined on all of $\evil(n-1,k-1)$ and contribute ${|\!\evil(n-1,k-1)|}$ permutations.

The operator $\psi_s$ is only defined on permutations that end with a copy of $e_t$ for some $t \ge 1$. The intersection with $\evil(n-1,k)$ are those permutations that are not in $\psi_p(\evil(n-2,k))$ or $\psi_q(\evil(n-2,k)).$ This contributes $|\!\evil(n-1,k)|-2|\!\evil(n-2,k)|$ permutations.

This decomposition of $\evil(n,k)$ combinatorially proves 

\begin{align*}
|\!\evil(n,k)| &= 2|\!\evil(n-1,k)|+|\!\evil(n-1,k-1)| + (|\!\evil(n-1,k)|-2|\!\evil(n-2,k)|)\\ 
&=3|\!\evil(n-1,k)|+|\!\evil(n-1,k-1)|-2|\!\evil(n-2,k)|.
\end{align*}

\end{proof}

\section{Table of examples}\label{sec:examples}
We give examples of the bijection between rectangular and evil-avoiding permutations below. 

\begin{longtable}[c]{| c | c | c | c |} 
 \caption{The bijection for all rectangular permutations in $S_1,S_2,S_3,S_4$ and three randomly selected permutations in $S_5$.}\label{table:ex}\\

 \hline
 Rect. perm. & Word in $L_{\rect}$ & Word in $L_{\evil}$ & Evil-av. perm.\\
 \hline
 \endfirsthead

 \hline
 Rect. perm. & Word in $L_{\rect}$ & Word in $L_{\evil}$ & Evil-av. perm.\\
 \hline
 \endhead

 \hline
 \endfoot

 \hline

 \endlastfoot

 $1$ & $1$ & $s$ & $1$ \\
 $12$ & $11$ & $ss$ & $12$ \\
 $21$ & $d1$ & $rs$ & $21$ \\
 $123$ & $111$ & $sss$ & $123$ \\
 $132$ & $1d1$ & $srs$ & $312$ \\
 $213$ & $d11$ & $rss$ & $231$ \\
 $231$ & $ud1$ & $qrs$ & $213$ \\
 $312$ & $2d1$ & $prs$ & $132$ \\
 $321$ & $dd1$ & $rrs$ & $321$ \\
 $1234$ & $1111$ & $ssss$ & $1234$ \\
 $1243$ & $11d1$ & $ssrs$ & $4123$ \\
 $1324$ & $1d11$ & $srss$ & $3412$ \\
 $1342$ & $1ud1$ & $qsrs$ & $3142$ \\
 $1423$ & $12d1$ & $psrs$ & $1423$ \\
 $1432$ & $1dd1$ & $rsrs$ & $4231$ \\
 $2134$ & $d111$ & $rsss$ & $2341$ \\
 $2143$ & $d1d1$ & $srrs$ & $4312$ \\
 $2314$ & $ud11$ & $qrss$ & $2134$ \\
 $2341$ & $uud1$ & $qqrs$ & $2314$ \\
 $3124$ & $2d11$ & $prss$ & $1342$ \\
 $3142$ & $2ud1$ & $qprs$ & $3124$ \\
 $3214$ & $dd11$ & $rrss$ & $3421$ \\
 $3241$ & $dud1$ & $qrrs$ & $2143$ \\
 $3412$ & $u2d1$ & $pqrs$ & $1324$ \\
 $3421$ & $udd1$ & $rqrs$ & $3241$ \\
 $4123$ & $22d1$ & $pprs$ & $1243$ \\
 $4132$ & $2dd1$ & $rprs$ & $2431$ \\
 $4312$ & $d2d1$ & $prrs$ & $1342$ \\
 $4321$ & $ddd1$ & $rrrs$ & $4321$ \\
 $15423$ & $1d2d1$ & $prsrs$ & $15342$ \\
 $43125$ & $d2d11$ & $prrss$ & $14532$ \\
 $51432$ & $2ddd1$ & $rrprs$ & $35421$ 
 \end{longtable}

\section{The 1-almost-increasing permutations}\label{sec:bael}

We consider a class of permutations studied by Knuth \cite{Knu} and Elizalde \cite{Eli}. Kim and Williams \cite{KW} mentioned these permutations as Wilf-equivalent to evil-avoiding permutations and stated it would be particularly interesting to find a length-preserving bijection.

\begin{defin}
\textup{\cite{Eli}} Let $\An$ be the permutations $\pi \in S_n$ so for every $i \in \{1, 2, \dots, n\}$ there is at most one $j \le i$ with $\pi_j > i.$ Equivalently, $\pi$ avoids $4321, 4312, 3421,$ and $3412$. These permutations $\pi$ are called \emph{1-almost-increasing.} \footnote{We suggest the name \emph{bael-avoiding} because if we replace 4 with B, 3 with A, 2 with E, and 1 with L, then the avoided patterns correspond to BAEL, BALE, ABEL, and ABLE. Bael is a name for a demon and for a type of tree native to the Indian subcontinent.}
\end{defin}

This class of permutations is motivated by sorting algorithms, but it also arises as a property of products of pattern-avoiding permutations.

\begin{prop}
If $\tau_1$ and $\tau_2$ avoid 132 and 123, then $\tau_1 \circ \tau_2$ is 1-almost-increasing.
\end{prop}
\begin{proof}
By inspection, this is true in $S_4$: The permutations in $S_4$ avoiding 132 and 123 are \newline $\{3214, 3241, 3412, 3421, 4213, 4231, 4312, 4321 \}.$ The products of two such permutations cover all permutations in $S_4$ except $4321, 4312, 3421,$ and $3412.$ 

Any pattern of $\tau_1 \circ \tau_2$ in positions $\{a,b,c,d\}$ would produce that pattern in the size $4$ permutation $\tau_1' \circ \tau_2'$, where $\tau_2'$ is the restriction of $\tau_2$ to $\{a,b,c,d\}$ and $\tau_1'$ is the restriction of $\tau_1$ to $\tau_2(\{a,b,c,d\})$. 

Since the reduction of $\tau_1' \circ \tau_2'$ cannot be in $\{4321, 4312, 3421, 3412\},$ $\tau_1 \circ \tau_2$ avoids these patterns. 
\end{proof}

These types of product constructions are studied more by the author in \cite{Tung}.

\begin{theorem}\label{thm:bael-equivalent}
The 1-almost-increasing permutations are Wilf-equivalent to evil-avoiding permutations. That is, for all $n$, $|\An|=|\!\evil(n)|.$
\end{theorem}

These collections of patterns are not trivially Wilf-equivalent by a dihedral symmetry. We prove Wilf-equivalence between 1-almost-increasing and rectangular permutations with a bijection using the regular language $L_{\rect}$.

\begin{theorem}
For $n\ge 1,$ $$\Ao = \rho_{1,1}(\An) \cup \rho_{1,2}(\An) \cup \rho_{2,1}(\An) \cup \rho_{2,2}(\An).$$

\end{theorem}

\begin{proof}
We prove containment both ways. 

First, every element $\pi$ of $\Ao$ must have $\pi_1 \in \{1,2\}$ or $\pi_2 \in \{1,2\}$, since otherwise the values $\pi_1$ and $\pi_2$ are greater than $1$ and $2$, so one of the forbidden patterns appears in positions $\{1,2,\pi^{-1}_1,\pi^{-1}_2\}.$ Thus, $$\Ao \subset \rho_{1,1}(\An) \cup \rho_{1,2}(\An) \cup \rho_{2,1}(\An) \cup \rho_{2,2}(\An).$$

Second, none of the forbidden patterns has a 1 or 2 in the first two positions. Thus, inserting a 1 or 2 in positions 1 or 2 in $\pi \in \An$ cannot create a forbidden $4321, 4312, 3421,$ or $3412$. Thus, $$\rho_{1,1}(\An) \cup \rho_{12}(\An) \cup \rho_{2,1}(\An) \cup \rho_{2,2}(\An) \subset \Ao.$$
\end{proof}

Let the permutation in $S_1$ be encoded as $\rho_{1,1}$. 

\begin{theorem}
Every 1-almost-increasing permutation of size $n \ge 1$ can be written uniquely as a string of size $n$ in $\rho_{1,1}, \rho_{1,2}, \rho_{2,1}, \rho_{2,2}$ ending in $\rho_{1,1}$ with no substrings $\rho_{2,1}\rho_{1,1}$ or $\rho_{2,2}\rho_{1,1}.$ 
\end{theorem}
\begin{proof}
By the previous theorem, every 1-almost-increasing permutation can be written as a string of lengthening operators in $\{\rho_{1,1},\rho_{1,2},\rho_{2,1},\rho_{2,2}\}$. There are redundant encodings. If $\pi_1=1$, then $\rho_{1,1}(\pi) = \rho_{2,2}(\pi)$ and $\rho_{1,2}(\pi) = \rho_{2,1}(\pi)$, which imply $\{\pi_1,\pi_2\}=\{1,2\}$. Otherwise, $\pi \in \Ao$ is in the image of precisely one of $\rho_{1,1}, \rho_{1,2}, \rho_{2,1},$ and $\rho_{2,2}.$ 

Suppose encodings without $\rho_{2,2}\rho_{1,1}$ or $\rho_{2,1}\rho_{1,1}$ are not unique, and that $\pi$ is a shortest example with two encodings. If $\pi$ were in the image of just one of $\{\rho_{1,1},\rho_{1,2},\rho_{2,1},\rho_{2,2}\}$, then $\pi = \rho_{i,j} \tau$ for some shorter $\tau$ that also doesn't have a unique encoding, contradicting that $\pi$ is a shortest example. 

If $\pi_1=2$ and $\pi_2=1$, then there is a permutation $\tau$ so that $\pi = \rho_{2,1}\rho_{1,1}\tau = \rho_{1,2}\rho_{1,1}\tau.$ Since $\rho_{1,1}\tau$ is shorter than $\pi$, it must have a unique encoding $\rho_{1,1}w$, and since $\rho_{2,1}\rho_{1,1}$ is forbidden, $\pi$ has a unique encoding in this language, namely $\rho_{1,2}\rho_{1,1}w$.

If $\pi_1=1$ and $\pi_2=2$, then $\pi$ starts with an increasing consecutive sequence $[1~2~\cdots~(t-1)~t~\cdots]$. Then since $\rho_{2,2}\rho_{1,1}$ is forbidden, all encodings for $\pi$ must start with $\rho_{1,1},$ so if $\pi = \rho_{1,1}\tau$, there must be multiple encodings for $\tau$ which is shorter than $\pi,$ a contradiction.

Therefore, encodings without $\rho_{2,2}\rho_{1,1}$ or $\rho_{2,1}\rho_{1,1}$ are unique.
\end{proof}

Let $L_{\Ai}$ be this language over the alphabet $\{\rho_{1,1},\rho_{1,2}, \rho_{2,1}, \rho_{2,2}\}$. The languages $L_{\Ai}$ and $L_{\rect}$ are isomorphic by the length-preserving substitution $\rho_{1,1} \mapsto 1, \rho_{2,1} \mapsto 2, \rho_{2,2} \mapsto u, \rho_{1,2} \mapsto d.$

Thus, 1-almost-increasing permutations and evil-avoiding permutations are Wilf-equivalent, and Theorem \ref{thm:bael-equivalent} is proven.

\section{Paths}\label{sec:paths}

There are several other families of combinatorial objects counted by the same sequence A006012 \cite{OEIS}. One is walks of length $2n-2$ starting and ending in the middle vertex of $P_7$, a path graph with $7$ vertices. Let the vertices of $P_7$ be $\{v_1,v_2,v_3,v_4,v_5,v_6,v_7\}$.

\begin{theorem}
There is an explicit bijection between walks of length $2n-2$ in $P_7$ starting and ending at $v_4$ and rectangular permutations of length $n$.
\end{theorem}
Elizalde \cite{Eli} also produced a bijection between walks of length $2n-2$ starting and ending at the middle vertex of $P_7$ and rectangular permutations in $S_n$ by bijecting rectangular permutations with words of length $n-1$ in a regular language on the alphabet $\{E,W,R,L\}$ ending in $W$ or $E$ with no $LE$ or $RW$. Our proof is different.

\begin{proof}
We produce a bijection between these paths and words of length $n-1$ in a language that is $L_{\rect}$ with the trailing $1$ of each word truncated, hence to $L_{\rect}$ and to rectangular permutations. We create a regular language encoding paths using a discrete finite automaton (DFA).

Instead of using all 7 vertices of $P_7$ as states, we consider taking steps two by two. By parity, after starting at the middle vertex $v_4$, a walk must reach a vertex in $\{v_2,v_4,v_6\}$ after an even number of steps. Elements of this set will correspond to the states of a DFA with alphabet $\{\LL, \LR, \RL, \RR\}$.

$$
\begin{tikzpicture}
\node[state] (v2) {$v_2$};
\node[state, initial, initial where=above, accepting, right of=v2] (v4) {$v_4$};
\node[state, right of=v4] (v6) {$v_6$};
\draw (v2) edge[loop below] node{LR,RL} (v2)
(v2) edge[bend left, above] node{RR} (v4)
(v4) edge[bend left, below] node{LL} (v2)
(v4) edge[loop below] node{LR,RL} (v4)
(v6) edge[loop below] node{LR,RL} (v6)
(v4) edge[bend left, above] node{RR} (v6)
(v6) edge[bend left, below] node{LL} (v4);
\end{tikzpicture}
$$

The extra arrow into $v_4$ indicates that it is the initial state and the double circle indicates that it is the only accepting state. The edges are labeled by the pair or pairs of steps that move from one state to the next. The character L indicates a leftward step and the character R indicates a rightward step. We can convert this DFA to a regular expression generating the language, producing $$(\textrm{LR}|\textrm{RL}|\textrm{LL}(\textrm{LR}|\textrm{RL})^*\textrm{RR}|\textrm{RR}(\textrm{LR}|\textrm{RL})^*\textrm{LL})^*.$$ There is no way of substituting the pairs LL, LR, RL, and RR to $\{1,2,d,u\}$ to produce a language similar to $L_{\rect}.$ To see why, we observe that the number of LL symbols must be equal to the number of RR symbols, and no analogous restriction occurs in $L_{\rect}.$ Instead, we relabel the state transitions as follows:

$$\begin{tikzpicture}
\node[state] (v2) {$v_2$};
\node[state, initial, initial where=above, accepting, right of=v2] (v4) {$v_4$};
\node[state, right of=v4] (v6) {$v_6$};
\draw (v2) edge[loop below] node{2,u} (v2)
(v2) edge[bend left, above] node{d} (v4)
(v4) edge[bend left, below] node{u} (v2)
(v4) edge[loop below] node{1,d} (v4)
(v6) edge[loop below] node{2,u} (v6)
(v4) edge[bend left, above] node{2} (v6)
(v6) edge[bend left, below] node{d} (v4);
\end{tikzpicture}$$

If we convert this to a regular expression, we get 
$$(1|d|u(2|u)^*d|2(2|u)^*d)^* = (1 | (2|u)^*d)^*.$$
The language of this expression is all strings in $\{1,2,d,u\}$ with no $21$ or $u1$ substring that do not end in $2$ or $u$. This is almost the same as $L_{\rect}$. We can make a bijection by adding a terminal $1$. This proves the theorem.

\end{proof}

\section{Future directions}\label{sec:future-dir}
We discuss possible future directions, both algebraic and enumerative.

In the algebraic direction, the bijection between rectangular and evil-avoiding permutations may help us derive the inhomogeneous TASEP steady-state probabilities corresponding to rectangular permutations. Kim and Williams express these steady-state probabilities corresponding to evil-avoiding permutations as a ``trivial factor” times a product of (double) Schubert polynomials (for a more precise statement, see Theorem 1.11 in \cite{KW}), so there is hope that the rectangular permutations may give rise to nice steady-state expressions. 

Going in an enumerative direction, one could try to biject other pairs of objects counted by A006012, as suggested by Williams in personal communication \cite{pc}. Two sets of objects seem likely to shed light on evil-avoiding permutations, either because they are permutation classes or are clearly recursively constructed. The last was included in \cite{KW} as particularly interesting.
\begin{enumerate}
    \item paths of length $2n$ with $n=0$ starting at the initial node on the path graph with 7 vertices,
    \item permutations on $[n]$ with no subsequence $abcd$ such that (i) $bc$ are adjacent in position and (ii) $\max(a, c) < \min(b, d)$.
\end{enumerate}

From another enumerative perspective, we may want to consider which sets of permutations can be bijected using similar techniques as those used here. Permutations avoiding $\{4312,4231\}$, $\{4312,4213\}$, $\{4231,4213\}$, $\{4213,4132\}$, and $\{4213,3214\}$ are known to be equinumerous. When they are triply graded by size, recoils, and descents, these five sets of permutations still appear equinumerous. It would be interesting to construct a bijection between any pair of these five sets of permutations, and particularly so if the bijection preserves the number of recoils and/or descents. Evil-avoiding and rectangular permutations also seem to be equinumerous when triply graded by size, recoils, and descents.

Here, we used regular languages to biject two particular families of permutations. This technique may be more widely applicable to other permutation families, but there is a limitation that the generating functions of regular languages must be rational (the sequences satisfy linear recurrence relations). Many families of pattern-avoiding permutations are known not to have rational generating functions, e.g., permutations avoiding any one pattern of size three are counted by Catalan numbers, and the generating function for Catalan numbers is algebraic but not rational. 
However, some pattern-avoiding permutation families have rational generating functions. For these permutation families, it would be interesting to check if there are natural bijections with regular languages. If so, then the techniques of this paper might be used to establish many other nontrivial Wilf-equivalences.

\section{Acknowledgments}\label{sec:acknowledgments}
The author thanks Joseph Gallian and the University of Minnesota Duluth REU during which this research was conducted, Amanda Burcroff, Andrew Kwon, and Mitchell Lee for valuable feedback on the paper, and all participants and advisors for helpful discussions. This work was done with support from
Jane Street Capital, the NSA (grant number H98230-22-1-0015), the NSF (grant number DMS-2052036), and Harvard University.

\printbibliography

\end{document}